\documentclass[11pt,reqno]{amsart}
\usepackage{amsfonts}
\usepackage{amsfonts}
\usepackage{amsfonts}
\usepackage{amsfonts}
\usepackage{amsfonts}
\usepackage{amsfonts}
\usepackage{amsfonts}
\usepackage{amsmath}
\usepackage{amsfonts}
\usepackage{mathrsfs}
\usepackage{amsfonts}
\usepackage{amssymb,amsfonts,amsmath, psfrag,eepic,colordvi,epsfig}
\topskip 
\numberwithin{equation}{section}
\newtheorem{theorem}{Theorem}[section]

\newtheorem{conjecture}[theorem]{Conjecture}

\newtheorem{lemma}[theorem]{Lemma}

\begin{document}
\parskip 7pt

\pagenumbering{arabic}
\def\sof{\hfill\rule{2mm}{2mm}}
\def\ls{\leq}
\def\gs{\geq}
\def\SS{\mathcal S}
\def\qq{{\bold q}}
\def\MM{\mathcal M}
\def\TT{\mathcal T}
\def\EE{\mathcal E}
\def\lsp{\mbox{lsp}}
\def\rsp{\mbox{rsp}}
\def\pf{\noindent {\it Proof.} }
\def\mp{\mbox{pyramid}}
\def\mb{\mbox{block}}
\def\mc{\mbox{cross}}
\def\qed{\hfill \rule{4pt}{7pt}}
\def\pf{\noindent {\it Proof.} }
\textheight=22cm

{\Large
\begin{center}
A proof of a conjecture of Mao on
 Beck's partition  statistics modulo 8
\end{center}
}

\begin{center}

$^{1}$Renrong Mao  and $^{2}$Ernest X.W. Xia

Department of Mathematics,\\
 Soochow University, \\
 Suzhou, 215006,
People's Republic of China\\[6pt]

$^2$School of Mathematical Sciences, \\
  Suzhou University of Science and
Technology, \\
 Suzhou,  215009, Jiangsu Province,
 P. R. China

Email: rrmao@suda.edu.cn,
  ernestxwxia@163.com

\end{center}

\noindent {\bf Abstract.}
  Beck introduced
two partition  statistics
   $NT(r,m,n)$ and $M_{\omega}(r,m,n)$,
    which denote
     the total number of parts
      in the partition
      of $n$ with rank congruent
       to $r$ modulo $m$ and
    the total number of ones
      in
    the partition
      of $n$ with crank congruent
       to $r$ modulo $m$, respectively. In recent
        years, a number of congruences
         and identities on $NT(r,m,n)$ and $M_{\omega}(r,m,n)$
         for some small $m
          $ have been established.
           In this paper, we prove an identity
            on $NT(r,8,n)$ and
            $M_{\omega}(r,4,n)$
             which confirm a conjecture
              given by Mao.



 \allowdisplaybreaks

\section{Introduction}

 A partition
 $\pi=(\pi_1,\pi_2,\ldots,\pi_k)$
   of a positive integer
 $n$ is a   sequence
  of positive integers $\pi_1\geq
   \pi_2\geq \cdots
  \geq \pi_k>0$ such that $\pi_1
  +\pi_2+\cdots
  +\pi_k=n$. The $\pi_i$
   are called the parts of the partition
    \cite{Andrews-1976}.  In this paper,
 we shall write  $\pi \vdash n$
    if  $\pi$  is  a partition of  $n$.
 Let  $\#(\pi)$ and $\lambda(\pi)$
 denote the total number of parts of  $\pi$
  and the largest part of
$\pi$, respectively.
 As usual,
   let $p(n)$ denote the number of
    partitions of $n$ and set $p(0)=1$.
 The following three famous congruences for
 $p(n)$ were discovered by Ramanujan \cite{Ramanujan-1}:
 \begin{align*}
     p(5n + 4) &\equiv  0 \pmod 5,   \\
     p(7n + 5) & \equiv 0
      \pmod 7,    \\
      p(11n + 6) & \equiv 0
       \pmod {11}.
    \end{align*}

In order to explain the above
 three   congruences
combinatorially,
 two   partition statistics,
    rank and crank, were defined  by
 Dyson \cite{Dyson}, and Andrews and Garvan
 \cite{Andrews-2}, respectively.
 In 1944,  Dyson \cite{Dyson} defined
 the rank of a partition to be the largest part
 minus the number of parts, i.e.,
\[
{\rm rank}(\pi):=\lambda(\pi)-\#(\pi).
\]
For example,
 the rank of the partition
 $2+1+1+1$ is $2-4=-2$.
 In 1988, Andrews and Garvan
   \cite{Andrews-2}
defined the crank by
\begin{align*}
\operatorname{crank}(\pi):=\left\{\begin{array}{ll}
\lambda(\pi), & \text { if } \omega(\pi)=0, \\
\mu(\pi)-\omega(\pi), & \text { if } \omega(\pi)>0,
\end{array}\right.
\end{align*}
where  $\omega(\pi)$ counts the number
 of ones in $\pi$ and $\mu(\pi)$ counts the number
 of parts larger than $\omega(\pi)$.
  For example, the crank of the partition
   $2+1+1+1$ is $0-3=-3$
    while the crank of the partition
   $4+2+2$ is $4$.

 Recently, Andrews
 \cite{Andrews} mentioned that George
  Beck defined
  two partition  statistics
   $NT(r,m,n)$ and $M_{\omega}(r,m,n)$,
    which count
     the total number of parts in the partition
      of $n$ with rank congruent
       to $r$ modulo $m$ and
    the total number of ones
      in
    the partition
      of $n$ with crank congruent
       to $r$ modulo $m$, respectively.
Utilizing
   the results on rank differences
obtained in \cite{Atkin},
  Andrews  \cite{Andrews}   proved the following
 interesting congruences  conjectured
  by Beck:
\[
\sum_{m=1}^4
 m NT(m,5,5n+1) \equiv \sum_{m=1}^4
 m NT(m,5,5n+4) \equiv
0 \pmod 5
\]
and for $i\in\{1,5\}$,
\begin{align*}
 NT(1,7,7n+i)&- NT(6,7,7n+i)
 + NT(2,7,7n+i)- NT(5,7,7n+i)\nonumber\\[6pt]
 &- NT(3,7,7n+i)+ NT(6,7,7n+i) \equiv
0 \pmod 7.
\end{align*}
Motivated  by Andrews' work,
 a number of  identities and
   congruences on $NT(r,m,n)$
    and $M_{\omega}(r,m,n)$ and their
     variations have been
     proved; see for example
     \cite{Chan,Chern-1,Chern-2,Chern-3,Du,Du-1,Liuxin-jin,Lin,mao-1,mao-2,mao-3,Mao-Xia,Xuan,Yao}.
      Very recently, Mao  \cite{mao-1}
       proved some
      identities
        on   the total
number of parts functions associated to ranks of overpartition.
 At the end of his paper \cite{mao-1}, Mao
 conjectured five identities
 on $NT(r,m,n)$ and $M_{\omega}(r,m,n)$
  and three of them were proved by
  Jin, Liu and Xia \cite{Liuxin-jin}, and Mao
   and Xia \cite{Mao-Xia}. The rest two conjectural
    identities are listed as follows.

\begin{conjecture}\label{c8}
    For $n\geq 0$,
    \begin{align*}
    NT(2,8,4n)-NT(6,8,4n) &= M_{\omega}(1,4,4n)-M_{\omega}(3,4,4n),\\
        NT(6,8,4n+2)- NT(2,8,4n+2)& = M_{\omega}(1,4,4n+2)-M_{\omega}(3,4,4n+2).
  \end{align*}

\end{conjecture}

The aim of this paper is to present a proof
 of the following theorem
  which implies
Conjecture \ref{c8}.

\begin{theorem}\label{thmain}
    For $n\geq 0$,
    \begin{align*}
       NT(2,8,2n)- NT(6,8,2n)&
       =(-1)^n\left(
       M_{\omega}(1,4,2n)
       -M_{\omega}(3,4,2n)\right).
    \end{align*}
\end{theorem}

\section{The Generating
 Function  for $ M_{\omega}(1,4,2n)
       -M_{\omega}(3,4,2n)$}

This aim of this section is to establish
  a generating function for $ M_{\omega}(1,4,2n)
       -M_{\omega}(3,4,2n)$.

       Recall some   $q$-series notations
\begin{align*}
(a)_\infty:=(a;q)_\infty:&=\prod_{n=0}^\infty
(1-aq^n),\\[6pt]
(a_1,a_2,\ldots,a_k)_\infty :=(a_1,a_2,\ldots,a_k;q)_\infty
:&=(a_1;q)_\infty
 (a_2;q)_\infty \cdots (a_k;q)_\infty,\\[6pt]
 [a_1,a_2,\ldots,a_k]_\infty:=[a_1,a_2,\ldots,a_k;q]_\infty
  :&=(a_1,q/a_1,a_2,q/a_2,\ldots,
  a_k, q/a_k;q)_\infty,\\[6pt]
  J_{r,m}:&= (q^r, q^{m-r},q^m;q^m)_\infty,
  \end{align*}
  and
  \[
 J_m
  :=(q^m;q^m)_\infty.
  \]

  \begin{lemma} \label{L-1}
        We have
        \begin{align}\label{v-1}
            &\sum_{n\geq 0} (M_{\omega}(1,4,4n) -M_{\omega}
            (3,4,4n))q^n
            \nonumber\\[6pt]
            =& \frac{1}{4J_1}A_0(q)B_0(q)(1-\varphi(q)^2)
            +\frac{q}{J_1}\bigg(
            \frac{1}{4}A_2(q)B_2(q)(1-\varphi(q)^2)
            -A_2(q)B_1(q)\psi(q)^2 \nonumber\\[6pt]
            &
            +(A_0(q)B_2(q)
            +A_2(q)B_0(q))\psi(q^2)^2 \bigg)
            -\frac{q^2}{J_1}A_0(q)
            B_3(q)\psi(q)^2
        \end{align}
        and
        \begin{align}\label{v-2}
            &\sum_{n\geq 0} (M_{\omega}(1,4,4n+2)-M_{\omega}
            (3,4,4n+2))q^n
            \nonumber\\[6pt]
            =& \frac{1}{4J_1}(A_0(q)B_2(q)
            +A_2(q)B_0(q))(\varphi(q)^2-1)
            +\frac{1}{J_1}A_0(q)(
            B_1(q)\psi(q)^2 -B_0(q)\psi(q^2)^2)
            \nonumber\\[6pt]
            & -\frac{q}{J_1}A_2(q)B_2(q)\psi(q^2)^2
            +\frac{q^2}{J_1}A_2(q)B_3(q)\psi(q)^2,
        \end{align}
        where
        \begin{align}
            A_0(q):&=\frac{(q^2,q^{6},q^{8};q^{8}
                )_\infty
            }{(-q,-q^{7};q^{8})_\infty}, \
            A_2(q):=\frac{(q^2,q^{6},q^{8};
            q^{8})_\infty
            }{(-q^{3},-q^{5};
            q^{8})_\infty},\
            B_0(q):=\frac{(q^{6},q^{10},q^{16};q^{16})_\infty
            }{(-q^{3},-q^{13};q^{16})_\infty}
            ,\label{v-3}\\[6pt]
            B_1(q):&=\frac{(q^{2},q^{14},q^{16};q^{16})_\infty
            }{(-q,-q^{15};q^{16})_\infty}
            ,\  B_2(q):=\frac{(q^{6},q^{10},q^{16}
                ;q^{16})_\infty
            }{(-q^{5},-q^{11};
            q^{16})_\infty},\
            B_3(q):=\frac{(q^{2},q^{14},q^{16};q^{16})_\infty
            }{(-q^{7},-q^{9};q^{16})_\infty}
            ,\label{v-4}\\[6pt]
            \varphi(q):&=\sum_{n=-\infty}^\infty
            q^{n^2}=
            \frac{J_2^{5}
            }{J_1^2
                J_4^2},\quad
            \psi(q):=\sum_{n=0}^\infty
            q^{n(n+1)/2}=\frac{J_2^2
            }{J_1}. \label{v-5}
        \end{align}
    \end{lemma}

    \noindent{\it Proof.}  In \cite{Mao-Xia},
      Mao and Xia
     proved that
    \begin{align}\label{v-7}
        \sum_{n\geq 0} M_{\omega}(a,k,n)q^n =
        & \frac{1}{k} \sum_{j=0}^{k-1}
        \zeta_k^{-aj}\frac{J_1
        }{(\zeta_k^j q;q)_\infty
            (q/\zeta_k^j;q)_\infty}
        \left(\sum_{n=1}^\infty
        \frac{\zeta_k^{-j}q^n}{1-q^n\zeta_k^{-j}}
        -S(q)\right)\nonumber\\[6pt]
        =& T(q)
        +\frac{1}{k} \sum_{j=1}^{k-1}
        \zeta_k^{-aj}\frac{J_1
        }{(\zeta_k^j q;q)_\infty
            (q/\zeta_k^j;q)_\infty}
        \left(\sum_{n=1}^\infty
        \frac{\zeta_k^{-j}q^n}
        {1-q^n\zeta_k^{-j}}
        -S(q)\right),
    \end{align}
    where $\zeta_k=e^{2\pi i/k}$ and
    \begin{align}\label{v-8}
        T(q):=\frac{q}{k(1-q)J_1}, \qquad S(q):=\sum_{n=1}^\infty
        \frac{q^{n+1}}{1-q^{n+1}}.
    \end{align}
    It is easy to check that
    \begin{equation} \label{v-9}
        \frac{J_1
        }{(\zeta_4^j q;q)_\infty
            (q/\zeta_4^j;q)_\infty}=\left\{ \begin{aligned}
            &\frac{J_1J_2
            }{J_4},\qquad\quad\
            {\rm if }\ j=1,3, \\
            & \frac{J_1^3}{J_2^2},
            \ \  \qquad \qquad
            {\rm if}\ j=2.
        \end{aligned} \right.
    \end{equation}
    Moreover,
    \begin{align}\label{v-10}
        \sum_{n=1}^\infty
        \frac{\zeta_4^{-j}q^n}
        {1-q^n\zeta_4^{-j}}=\zeta_4^{-j}
        \sum_{n=1}^\infty
        \frac{q^n}{1-q^{4n}}+\zeta_4^{-2j}
        \sum_{n=1}^\infty \frac{q^{2n}}{1-q^{4n}}+\zeta_4^{-3j}
        \sum_{n=1}^\infty
        \frac{q^{3n}}{1-q^{4n}}+\sum_{n=1}^\infty
        \frac{q^{4n}}{1-q^{4n}}.
    \end{align}
    Setting $k=4$ and $a=1,3$ in \eqref{v-7}
    and employing \eqref{v-9}
    and \eqref{v-10},
    we deduce that
    \begin{align}\label{v-11-1}
        \sum_{n\geq 0} (M_{\omega}(1,4,n)-M_{\omega}
        (3,4,n))q^n=&
        \frac{J_1
            J_2}{J_4}
        \sum_{n=1}^\infty
        \frac{q^{3n}-q^n}{1-q^{4n}}
        =-\frac{J_1J_2}{J_4}
        \sum_{n=1}^\infty
        \frac{q^n}{1+q^{2n}}.
    \end{align}
    The following identity appears in
    Berndt's book \cite[(3.2.9)]{Berndt-2006}
    \begin{align}\label{v-12}
        \sum_{n=1}^\infty
        \frac{q^n}{1+q^{2n}}
        =\frac{1}{4}\left(
        \varphi(q)^2-1\right),
    \end{align}
    where $\varphi(q)$ is defined by
    \eqref{v-5}.
    Combining \eqref{v-11-1}
    and \eqref{v-12} yields
    \begin{align}\label{v-13}
        \sum_{n\geq 0} (M_{\omega}(1,4,n)-M_{\omega}
        (3,4,n))q^n
        =&\frac{1}{4}\frac{J_1J_2}{
            J_4}
        \left(1- \varphi(q)^2\right).
    \end{align}
    The following identity was proved by Xia and
    Yao \cite[Lemma 3.2, (3.4)]{Xia}
    \begin{align}\label{v-14}
        J_2
        = A_0(q^4)-q^2  A_2(q^4),
    \end{align}
    where $A_0(q)$ and $A_2(q)$ are defined
    by \eqref{v-3}.
    Lewis \cite[Corollary 6]{Lewis} proved
    that
    \begin{align}\label{v-15}
        J_1=&B_0(q^4) -qB_1(q^4)-q^2B_2(q^4)+q^7 B_3(q^4) ,
    \end{align}
    where $B_0(q)$, $B_1(q)$,
    $B_2(q)$ and $B_3(q)$ are
    defined by \eqref{v-3}
    and \eqref{v-4}.
    It follows from Entry 25 (v)
    and (vi) in Berndt's
    book \cite[p. 40]{Berndt} that
    \begin{align}\label{v-16}
        \varphi(q)^2=&\varphi(q^2)^2+4q\psi(q^4)^2 \nonumber
        \\[6pt]
        =&\varphi(q^4)^2+4q\psi(q^4)^2 +4q^2\psi(q^8)^2,
    \end{align}
    where $\psi(q)$ is defined by \eqref{v-5}.
    If we substitute \eqref{v-14}, \eqref{v-15}
    and \eqref{v-16} into \eqref{v-13},
    then extract
    those  terms in which the power of $q$
    is
    congruent to $i\ (i=0,2)$
    modulo 4, then
    divide  by $q^i$
    and
    replace  $q^4$ by $q$, we arrive at
    \eqref{v-1} and \eqref{v-2}. This completes the proof
    of Lemma \ref{L-1}.
    \qed

 \section{The  Generating
  Function for $NT(2,8,2n)- NT(6,8,2n)$}

 In this Section, we establish
  the generating function for $NT(2,8,2n)- NT(6,8,2n)$.

 \begin{theorem}\label{thrank}
    We have
    \begin{align}
        &\sum_{n=0}^{\infty}\left(NT(2,8,2n)
        -NT(6,8,2n)\right)q^n
        =R_1(q)+R_2(q),\label{thn8}
        \intertext{where}
            R_1(q):&=\left(\frac{[-q^{3};q^{8}]^2_\infty }{[-1;q^{8}]_\infty }-\frac{q^2[-q;q^{8}]^2_\infty }{[-q^{4};q^{8}]_\infty }\right)\times \frac{[q^{2};q^{8}]_\infty J_{8}^3}{2[-q^2,q^{3};q^{8}]_\infty J_1^2}\nonumber
        \intertext{and}
            R_2(q):&=\bigg(
        \frac{[q^2,-q^3,-q^3;q^{8}]_\infty }{2[-1,q,q,q,q^3;q^{8}]_\infty }
        -
        \frac{2[-q^3,-q^3,q^4;q^{8}]_\infty }{[-1,-1,-q^2,-q^2,-q^4;q^{8}]_\infty }
        \nonumber\\&\qquad
        -\frac{3q[q^2,-q^3-q^3;q^{8}]_\infty }{2[-1,q,q^{3},q^3,q^3;q^{8}]_\infty }
        -\frac{2q^2[-q,-q,q^4;q^{8}]_\infty }{[-1,-q^2,-q^2,-q^{4},-q^4;q^{8}]_\infty }
        \nonumber\\&\qquad-\frac{q^2[-q,-q,q^2;q^{8}]_\infty }{2[q,q,q,q^3,-q^{4};q^{8}]_\infty }
        +\frac{3q^3[-q,-q,q^2;q^{8}
        ]_\infty }{2[q,q^3,q^3,q^3,
        -q^4;q^{8}]_\infty }
        \bigg)\times
         \frac{[q^{2};q^{8}]^3_\infty
         J_{8}^5}{[-q^2,q^{3};
         q^{8}]_\infty J_1^2}.
         \nonumber
    \end{align}
 \end{theorem}

In order to prove Theorem \ref{thrank},
 we first prove some lemmas.

 \begin{lemma}
    We have
    \begin{align}\label{j1}
        q[-q^{2};q^{16}]_\infty-[-q^{6};q^{16}]_\infty&=-\frac{\left[q^{2}, q^{4}, q^{4}, q^{6}, q^{8} ; q^{16}\right]_{\infty}J_1J_{16}}{J_2^2},
        \\
        \label{x4}
        X\left(-q^{12};q^{16}\right)
        &   =\frac{1}{4}
        -\frac{[q^{4},q^{4},q^{8};q^{16}]_\infty
            J_{16}^2}{2[-q^{4};
            q^{16}]^2_\infty[
            -1,-q^{8};q^{16}]_\infty},
        \intertext{and}
        X\left(q^{22};q^{16}\right)
        &   =  \frac{7}{8}-\frac{3q^2
         [q^{4};q^{16}]^3_\infty
            J_{16}^2}{8[q^{6};
            q^{16}]^3_\infty[q^{2};q^{16}
            ]_\infty}+\frac{[q^{4};q^{16}]^3_\infty
            J_{16}^2}{8[q^{2};q^{16}]^3_\infty[q^{6};q^{16}]_\infty},\label{x22}
        \intertext{where}
        X(a;q):&=\sum_{n=0}^\infty\left(\frac{aq^{n}}{1-aq^n}-\frac{q^{n+1}/a}{1-q^{n+1}/a}\right)\label{defx}.
    \end{align}
 \end{lemma}
 \begin{proof}
    Equation \eqref{j1} follows immediately from the two identities (see \cite[Lemma 4.1]{four} and \cite[Eq.(5.5)]{four}, respectively):
    \begin{align}
        \frac{\left(q^{16} ; q^{16}\right)_{\infty}}{\left(q^{2} ; q^{2}\right)_{\infty}^{2}}\left(\left[-q^{6} ; q^{16}\right]_{\infty}+q\left[-q^{2} ; q^{16}\right]_{\infty}\right) &=   \frac{1}{J_1}\label{zp1}
        \intertext{and}
        \left[-q^{6} ; q^{16}\right]_{\infty}^{2}-q^{2}\left[-q^{2} ; q^{16}\right]_{\infty}^{2}&=\left[q^{2}, q^{4}, q^{4}, q^{6}, q^{8} ; q^{16}\right]_{\infty}.\nonumber
    \end{align}
    Recall \cite[Eq. (3.2)]{gen}:
    \begin{equation}\label{genchan}
        \begin{aligned}
            \frac{[a b, b c, c a]_{\infty}
                J_1^{2}}{[a, b, c, a b c]_{\infty}}=& 1+\sum_{k=0}^{\infty} \frac{a q^{k}}{1-a q^{k}}-\sum_{k=1}^{\infty} \frac{q^{k} / a}{1-q^{k} / a}+\sum_{k=0}^{\infty} \frac{b q^{k}}{1-b q^{k}} \\
            &-\sum_{k=1}^{\infty} \frac{q^{k} / b}{1-q^{k} / b}+\sum_{k=0}^{\infty} \frac{c q^{k}}{1-c q^{k}}-\sum_{k=1}^{\infty} \frac{q^{k} / c}{1-q^{k} / c} \\
            &-\sum_{k=0}^{\infty} \frac{a b c q^{k}}{1-a b c q^{k}}+\sum_{k=1}^{\infty} \frac{q^{k} / a b c}{1-q^{k} / a b c}.
        \end{aligned}
    \end{equation}
    Replacing $q$ by $q^{16}$, setting $a=b=-q^4,c=-q^8$ and noting that $$ X(-q^8;q^{16})=0,\,\,   X(-q^{16};q^{16})=\frac{1}{2},$$ we obtain
    \begin{align}\nonumber
        \frac{1}{2}+2X\left(-q^4;q^{16}\right)
        =\frac{[q^{4},q^{4},q^{8};q^{16}]_\infty
            J_{16}^2}{[-q^{4};q^{16}]^2_\infty[-1,-q^{8};q^{16}]_\infty},
    \end{align}
    which together with $   X(-q^4;q^{16})=-    X(-q^{12};q^{16})$  gives \eqref{x4}.

    Similarly, applying \eqref{genchan}, we find that
    \begin{align}
        1-3X\left(q^{22};q^{16}\right)
        -X\left(q^{-18};q^{16}\right)&
        =\frac{[q^{-12};q^{16}]^3_\infty
            J_{16}^2}{[q^{-6};q^{16}]^3_\infty[q^{-18};q^{16}]_\infty},\nonumber
        \\[6pt]
        4+3X\left(q^{-18};q^{16}\right)
        +X\left(q^{22};q^{16}\right)&
        =\frac{[q^{-36};q^{16}]^3_\infty
            J_{16}^2}{[q^{-18};q^{16}]^3_\infty[q^{-54};q^{16}]_\infty}.\nonumber
    \end{align}
    Then equation \eqref{x22} follows.
 \end{proof}
 Recall \cite[Lemma 2.3]{mao-2}:
 \begin{align}
    &\frac{(q)_{\infty}^{2}}{[b_1, b_2, b_3]_{\infty}}\left[    X(b_1;q)+   X(b_2;q)+   X(b_3;q)\right]
    \nonumber\\&=\frac{1}{[ b_2/b_1, b_3/b_1]_{\infty}}\sum_{n=-\infty}^{\infty} \frac{(-1)^{n} b_1 q^{3n(n+1) / 2}}{(1-b_1 q^{n})^2}\times\left(\frac{b_1^2q}{b_2b_3}\right)^n\nonumber\\&\quad
    +\frac{1}{[ b_1/b_2, b_3/b_2]_{\infty}}\sum_{n=-\infty}^{\infty} \frac{(-1)^{n} b_2 q^{3n(n+1) / 2}}{(1-b_2 q^{n})^2}\times\left(\frac{b_2^2q}{b_1b_3}\right)^n
    \nonumber\\&\quad
    +\frac{1}{[ b_1/b_3, b_2/b_3]_{\infty}}\sum_{n=-\infty}^{\infty} \frac{(-1)^{n} b_3 q^{3n(n+1) / 2}}{(1-b_3 q^{n})^2}\times\left(\frac{b_3^2q}{b_1b_2}\right)^n\label{chan0}
    \intertext{and}
    &\frac{1}{2[ b_1, b_2]_{\infty}}\left\{\sum_{n=1}^\infty\frac{-2q^{n}}{(1-q^{n})^2}+\mathcal{S}_1(b_1,b_2;q)\left[2-\mathcal{S}_1(b_1,b_2;q)\right]-\mathcal{S}_2(b_1,b_2;q)\right\}\nonumber\\&=
    \frac{1}{[ b_1, b_2]_{\infty}}\sum_{\mbox{\tiny$\begin{array}{c} n=-\infty\\ n\neq0\end{array}$}}^\infty \frac{(-1)^{n} q^{3n(n+1) / 2}}{(1- q^{n})^2}\times\left(\frac{q}{b_1b_2}\right)^n\nonumber\\&\quad+\frac{1}{[ b_2/b_1, 1/b_1]_{\infty}}\sum_{n=-\infty}^{\infty} \frac{(-1)^{n} b_1 q^{3n(n+1) / 2}}{(1-b_1 q^{n})^2}\times\left(\frac{b_1^2q}{b_2}\right)^n\nonumber\\&\quad
    +\frac{1}{[ b_1/b_2, 1/b_2]_{\infty}}\sum_{n=-\infty}^{\infty} \frac{(-1)^{n} b_2 q^{3n(n+1) / 2}}{(1-b_2 q^{n})^2}\times\left(\frac{b_2^2q}{b_1}\right)^n
    ,\label{chan1}
 \end{align}
 where
 \begin{align}
    \mathcal{S}_1(b_1,b_2;q):&= X(b_1;q)+   X(b_2;q)\label{defs1}
    \intertext{and}
    \mathcal{S}_2(b_1,b_2;q):&=\sum_{n=0}^\infty
    \bigg(\frac{2b_1q^{n}-b_1^2q^{2n}}{(1-b_1q^{n})^2}
    +\frac{q^{2n+2}/b_1^2}{(1-q^{n+1}/b_1)^2}+\frac{2b_2q^{n}-b_2^2q^{2n}}{(1-b_2q^{n})^2}+\frac{q^{2n+2}/b_2^2}{(1-q^{n+1}/b_2)^2}
    \bigg)\label{defs2}.
 \end{align}
 Applying \eqref{chan0} and \eqref{chan1}, we obtain the following.
 \begin{lemma}
    We have
    \begin{align}\label{41}
        &   \sum_{\mbox{\tiny$\begin{array}{c} n=-\infty\\ n\neq0\end{array}$}}^\infty\frac{(-1)^{n}q^{\frac{n(3n+1)}{2}+n}}{(1-q^{4n})^2}\nonumber\\&=
        -\frac{J_1}{J_{16}} \sum_{\mbox{\tiny$\begin{array}{c} n=-\infty\end{array}$}}^\infty\frac{(-1)^nq^{24n^2+72n+53}}{(1+q^{16n+22})^2}
        \nonumber\\&\quad+\frac{\left[X\left(q^{12};q^{16}\right)+X\left(-q^{22};q^{16}\right)\right]
            \times\left[1-X\left(q^{12};q^{16}\right)-X\left(-q^{22};q^{16}\right)\right]}{2}\nonumber\\&\quad-\sum_{n=-\infty}^\infty\left(    \frac{q^{16n+12}}{2(1-q^{16n+12})^2}-
        \frac{q^{16n+22}}{2(1+q^{16n+22})^2}\right)-\sum_{n=1}^\infty\frac{q^{16n}}{(1-q^{16n})^2}\nonumber\\&\quad-
        \frac{q^3[-q^{2};q^{16}]_\infty J_{16}^2}{[-q^{6},q^8;q^{16}]_\infty}\times\left[X\left(q^4;q^{16}\right)+X\left(-q^{22};q^{16}\right)\right],\\
        \nonumber\\
        \label{42}
        &   \sum_{\mbox{\tiny$\begin{array}{c} n=-\infty\\ n\neq0\end{array}$}}^\infty\frac{(-1)^{n}q^{\frac{n(3n+1)}{2}+2n}}{(1-q^{4n})^2}\nonumber\\&=
        \frac{J_1}{J_{16}}  \sum_{\mbox{\tiny$\begin{array}{c} n=-\infty\end{array}$}}^\infty\frac{(-1)^nq^{24n^2+88n+75}}{(1+q^{16n+22})^2}
        \nonumber\\&\quad-\frac{\left[X\left(q^{12};q^{16}\right)+X\left(-q^{22};q^{16}\right)-2\right]
            \times\left[X\left(q^{12};q^{16}\right)+X\left(-q^{22};q^{16}\right)-1\right]}{2}\nonumber\\&\quad-\sum_{n=-\infty}^\infty\left(    \frac{q^{16n+12}}{2(1-q^{16n+12})^2}-
        \frac{q^{16n+22}}{2(1+q^{16n+22})^2}\right)-\sum_{n=1}^\infty\frac{q^{16n}}{(1-q^{16n})^2}\nonumber\\&\quad-
        \frac{q^3[-q^{2};q^{16}]_\infty J_{16}^2}{[-q^{6},q^8;q^{16}]_\infty}\times\left[X\left(-q^{22};q^{16}\right)+X\left(q^4;q^{16}\right)-1\right],\\
        \nonumber\\
        \label{49}
        &   \sum_{\mbox{\tiny$\begin{array}{c} n=-\infty\\ n\neq0\end{array}$}}^\infty\frac{(-1)^{n}q^{\frac{n(3n+1)}{2}+9n}}{(1-q^{4n})^2}\nonumber\\&=-
        \frac{J_1}{J_{16}}  \sum_{\mbox{\tiny$\begin{array}{c} n=-\infty\end{array}$}}^\infty\frac{(-1)^nq^{24n^2+104n+97}}{(1+q^{16n+22})^2}
        \nonumber\\&\quad-\frac{\left[X\left(q^{12};q^{16}\right)+X\left(-q^{22};q^{16}\right)-3\right]
            \times\left[X\left(q^{12};q^{16}\right)+X\left(-q^{22};q^{16}\right)-2\right]}{2}\nonumber\\&\quad-\sum_{n=-\infty}^\infty\left(    \frac{q^{16n+12}}{2(1-q^{16n+12})^2}-
        \frac{q^{16n+22}}{2(1+q^{16n+22})^2}\right)-\sum_{n=1}^\infty\frac{q^{16n}}{(1-q^{16n})^2}\nonumber\\&\quad-
        \frac{q^3[-q^{2};q^{16}]_\infty J_{16}^2}{[-q^{6},q^8;q^{16}]_\infty}\times\left[X\left(-q^{22};q^{16}\right)+X\left(q^4;q^{16}\right)-2\right],\\
        \nonumber\\
        \label{410}
        &   \sum_{\mbox{\tiny$\begin{array}{c} n=-\infty\\ n\neq0\end{array}$}}^\infty\frac{(-1)^{n}q^{\frac{n(3n+1)}{2}+10n}}{(1-q^{4n})^2}\nonumber\\&=
        \frac{J_1}{J_{16}}  \sum_{\mbox{\tiny$\begin{array}{c} n=-\infty\end{array}$}}^\infty\frac{(-1)^nq^{24n^2+56n+31}}{(1+q^{16n+22})^2}
        \nonumber\\&\quad-\frac{\left[X\left(q^{12};q^{16}\right)+X\left(-q^{22};q^{16}\right)\right]
            \times\left[X\left(q^{12};q^{16}\right)+X\left(-q^{22};q^{16}\right)+1\right]}{2}\nonumber\\&\quad-\sum_{n=-\infty}^\infty\left(    \frac{q^{16n+12}}{2(1-q^{16n+12})^2}-
        \frac{q^{16n+22}}{2(1+q^{16n+22})^2}\right)-\sum_{n=1}^\infty\frac{q^{16n}}{(1-q^{16n})^2}\nonumber\\&\quad-
        \frac{q^3[-q^{2};q^{16}]_\infty J_{16}^2}{[-q^{6},q^8;q^{16}]_\infty}\times\left[X\left(-q^{22};q^{16}\right)+X\left(q^{4};q^{16}\right)+1\right],\\
        \nonumber\\
        \label{411}
        &   \sum_{ n=-\infty}^\infty\frac{(-1)^{n}q^{\frac{n(3n+1)}{2}+n}}{(1+q^{4n})^2}\nonumber\\&=
        -\frac{J_1}{J_{16}} \sum_{\mbox{\tiny$\begin{array}{c} n=-\infty\end{array}$}}^\infty\frac{(-1)^nq^{24n^2+72n+53}}{(1-q^{16n+22})^2}
        \nonumber\\&\quad+
        \frac{[q^4,-q^{6};q^{16}]_\infty J_{16}^2}{[-1,-q^{4},q^{6};q^{16}]_\infty}\times\left[X\left(-q^{12};q^{16}\right)+X\left(q^{22};q^{16}\right)-\frac{1}{2}\right]
        \nonumber\\&\quad+
        \frac{q^3[-q^{2},q^4;q^{16}]_\infty J_{16}^2}{[q^6,-q^{8},-q^4;q^{16}]_\infty}\times\left[X\left(-q^{12};q^{16}\right)-X\left(q^{22};q^{16}\right)\right],\\
        \nonumber\\
        \label{422}
        &   \sum_{ n=-\infty}^\infty\frac{(-1)^{n}q^{\frac{n(3n+1)}{2}+2n}}{(1+q^{4n})^2}\nonumber\\&=
        \frac{J_1}{J_{16}}  \sum_{\mbox{\tiny$\begin{array}{c} n=-\infty\end{array}$}}^\infty\frac{(-1)^nq^{24n^2+40n+15}}{(1-q^{16n+10})^2}
        \nonumber\\&\quad+
        \frac{[q^{4},-q^{6};q^{16}]_\infty J_{16}^2}{[-1,q^{6},-q^{4};q^{16}]_\infty}\times\left[\frac{3}{2}-X\left(-q^{12};q^{16}\right)-X\left(q^{22};q^{16}\right)\right]
        \nonumber\\&\quad-
        \frac{q^3[-q^{2},q^4;q^{16}]_\infty J_{16}^2}{[q^6,-q^{8},-q^4;q^{16}]_\infty}\times\left[X\left(-q^{12};q^{16}\right)-X\left(q^{22};q^{16}\right)+1\right],\\
        \nonumber\\
        \label{499}
        &   \sum_{ n=-\infty}^\infty\frac{(-1)^{n}q^{\frac{n(3n+1)}{2}+9n}}{(1+q^{4n})^2}\nonumber\\&=-
        \frac{J_1}{J_{16}}  \sum_{\mbox{\tiny$\begin{array}{c} n=-\infty\end{array}$}}^\infty\frac{(-1)^nq^{24n^2+8n-15}}{(1-q^{16n-10})^2}
        \nonumber\\&\quad+
        \frac{[q^{4},-q^{6};q^{16}]_\infty J_{16}^2}{[-1,q^{6},-q^{4};q^{16}]_\infty}\times\left[X\left(-q^{12};q^{16}\right)+X\left(q^{22};q^{16}\right)-\frac{5}{2}\right]
        \nonumber\\&\quad-
        \frac{q^3[-q^{2},q^4;q^{16}]_\infty J_{16}^2}{[q^{6},-q^{8},-q^4;q^{16}]_\infty}\times\left[-X\left(-q^{12};q^{16}\right)+X\left(q^{22};q^{16}\right)-2\right],\\
        \nonumber\\
        \label{41010}
        &   \sum_{ n=-\infty}^\infty\frac{(-1)^{n}q^{\frac{n(3n+1)}{2}+10n}}{(1+q^{4n})^2}\nonumber\\&=-
        \frac{J_1}{J_{16}}  \sum_{\mbox{\tiny$\begin{array}{c} n=-\infty\end{array}$}}^\infty\frac{(-1)^nq^{24n^2+24n-13}}{(1-q^{16n-6})^2}
        \nonumber\\&\quad-
        \frac{[q^{4},-q^{6};q^{16}]_\infty J_{16}^2}{[-1,q^{6},-q^{4};q^{16}]_\infty}\times\left[X\left(-q^{12};q^{16}\right)+X\left(q^{22};q^{16}\right)+\frac{1}{2}\right]
        \nonumber\\&\quad-
        \frac{q^{3}[-q^{2},q^4;
        q^{16}]_\infty J_{16}^2}{
        [q^{6},-q^{8},-q^4;
        q^{16}]_\infty}\times
        \left[X\left(-q^{12};
        q^{16}\right)-X\left(q^{22}
        ;q^{16}\right)-1\right].
    \end{align}
 \end{lemma}

 \begin{proof}
    Split the series according to the summation index $n$ modulo $4$ to obtain
    \begin{align}
        \sum_{\mbox{\tiny$\begin{array}{c} n=-\infty\\ n\neq0\end{array}$}}^\infty\frac{(-1)^{n}q^{\frac{n(3n+1)}{2}+n}}{(1-q^{4n})^2}&=
        \sum_{\mbox{\tiny$\begin{array}{c} n=-\infty\\ n\neq0\end{array}$}}^\infty\frac{q^{24n^2+6n}}{(1-q^{16n})^2}
        -   \sum_{\mbox{\tiny$\begin{array}{c} n=-\infty\end{array}$}}^\infty\frac{q^{24n^2+18n+3}}{(1-q^{16n+4})^2}  \nonumber\\&\quad+
        \sum_{\mbox{\tiny$\begin{array}{c} n=-\infty\end{array}$}}^\infty\frac{q^{24n^2+30n+9}}{(1-q^{16n+8})^2}-
        \sum_{\mbox{\tiny$\begin{array}{c} n=-\infty\end{array}$}}^\infty\frac{q^{24n^2+42n+18}}{(1-q^{16n+12})^2}
        \nonumber\\&=:S_0-S_1+S_2 -S_3  .\label{s012}
    \end{align}
    Applying \eqref{chan0} with $(q, b_1, b_2,b_3)$ replaced by $(q^{16},q^{4},q^{8},-q^{22})$, multiplying by $$\frac{[q^{4},-q^{18};q^{16}]_\infty}{q}$$ on both sides of the resulting equation and simplifying yields
    \begin{align}\label{s12}
        S_1-S_2 &=\frac{q^3[-q^{2};q^{16}]_\infty(q^{16};q^{16})^2_\infty}{[-q^{6},q^8;q^{16}]_\infty}\times\left[X\left(q^4;q^{16}\right)+X\left(-q^{22};q^{16}\right)\right]\nonumber\\&\quad+\frac{[q^{4};q^{16}]_\infty}{[-q^{2};q^{16}]_\infty}    \sum_{\mbox{\tiny$\begin{array}{c} n=-\infty\end{array}$}}^\infty\frac{(-1)^nq^{24n^2+72n+53}}{(1+q^{16n+22})^2}.
    \end{align}
    Similarly, we apply \eqref{chan1} with $(q, b_1, b_2)$ replaced by $(q^{16},q^{12},-q^{22})$, multiply by $$[q^{12},-q^{22};q^{16}]_\infty$$ on both sides of the resulting equation and simplify to obtain
    \begin{align}\label{s03}
        S_0-S_3 &=\frac{\mathcal{S}_1(q^{12},-q^{22};q^{16})\left[2-\mathcal{S}_1(q^{12},-q^{22};q^{16})\right]-\mathcal{S}_2(q^{12},-q^{22};q^{16})}{2}
        \nonumber\\&\quad-\sum_{n=1}^\infty\frac{q^{16n}}{(1-q^{16n})^2}+\frac{[q^{4};q^{16}]_\infty}{[-q^{6};q^{16}]_\infty}   \sum_{\mbox{\tiny$\begin{array}{c} n=-\infty\end{array}$}}^\infty\frac{(-1)^nq^{24n^2+72n+54}}{(1+q^{16n+22})^2}.
    \end{align}
    By \eqref{defs2}, we have
    \begin{align}
        &\mathcal{S}_2(q^{12},-q^{22};q^{16})
        \nonumber\\&=\sum_{n=0}^\infty
        \bigg(\frac{2q^{16n+12}-q^{32n+24}}{(1-q^{16n+12})^2}
        +\frac{q^{32n+8}}{(1-q^{16n+4})^2}-\frac{2q^{16n+22}-q^{32n+44}}{(1+q^{16n+22})^2}+\frac{q^{32n-12}}{(1+q^{16n-6})^2}\bigg).    \nonumber
    \end{align}
    Note that
    \begin{align}
        &   \sum_{n=0}^\infty
        \bigg(\frac{2q^{16n+12}-q^{32n+24}}{(1-q^{16n+12})^2}
        +\frac{q^{32n+8}}{(1-q^{16n+4})^2}\bigg)
        \nonumber\\&=   \sum_{n=0}^\infty
        \bigg(\frac{q^{16n+12}}{(1-q^{16n+12})^2}+\frac{q^{16n+12}-q^{32n+24}}{(1-q^{16n+12})^2}
        +\frac{q^{32n+8}-q^{16n+4}}{(1-q^{16n+4})^2}    +\frac{q^{16n+4}}{(1-q^{16n+4})^2}\bigg)
        \nonumber\\&=   \sum_{n=-\infty}^\infty
        \frac{q^{16n+12}}{(1-q^{16n+12})^2}+X\left(q^{12};q^{16}\right).    \nonumber
    \end{align}
    With a similar argument, one can verify that
    \begin{align}
        &\sum_{n=0}^\infty
        \bigg(\frac{q^{32n-12}}{(1+q^{16n-6})^2}-\frac{2q^{16n+22}-q^{32n+44}}{(1+q^{16n+22})^2}\bigg)
        \nonumber\\&=X\left(-q^{22};q^{16}\right)   -\sum_{n=-\infty}^\infty
        \frac{q^{16n+22}}{(1+q^{16n+22})^2}.\nonumber
    \end{align}
    Then
    \begin{align}
        &\mathcal{S}_2(q^{12},-q^{22};q^{16})
        \nonumber\\&=X\left(-q^{22};q^{16}\right)+X\left(q^{12};q^{16}\right)   +\sum_{n=-\infty}^\infty\left(  \frac{q^{16n+12}}{(1-q^{16n+12})^2}-
        \frac{q^{16n+22}}{(1+q^{16n+22})^2}\right).\label{ss2}
    \end{align}
    Substituting \eqref{ss2} into \eqref{s03}, invoking \eqref{defs1} and simplifying gives
    \begin{align}\label{ss03}
        S_0-S_3 &=\frac{[q^{4};q^{16}]_\infty}{[-q^{6};q^{16}]_\infty}  \sum_{\mbox{\tiny$\begin{array}{c} n=-\infty\end{array}$}}^\infty\frac{(-1)^nq^{24n^2+72n+54}}{(1+q^{16n+22})^2}-\sum_{n=1}^\infty\frac{q^{16n}}{(1-q^{16n})^2}
        \nonumber\\&\quad+\frac{\left[X\left(q^{12};q^{16}\right)+X\left(-q^{22};q^{16}\right)\right]
            \times\left[1-X\left(q^{12};q^{16}\right)-X\left(-q^{22};q^{16}\right)\right]}{2}\nonumber\\&\quad-\sum_{n=-\infty}^\infty\left(    \frac{q^{16n+12}}{2(1-q^{16n+12})^2}-
        \frac{q^{16n+22}}{2(1+q^{16n+22})^2}\right).
    \end{align}
    Substitute \eqref{s12} and \eqref{ss03} into \eqref{s012} and rearrange to obtain
    \begin{align}
        &   \sum_{\mbox{\tiny$\begin{array}{c} n=-\infty\\ n\neq0\end{array}$}}^\infty\frac{(-1)^{n}q^{\frac{n(3n+1)}{2}+n}}{(1-q^{4n})^2}\nonumber\\&=
        \left(  \frac{q[q^{4};q^{16}]_\infty}{[-q^{6};q^{16}]_\infty}-\frac{[q^{4};q^{16}]_\infty}{[-q^{2};q^{16}]_\infty}\right)   \sum_{\mbox{\tiny$\begin{array}{c} n=-\infty\end{array}$}}^\infty\frac{(-1)^nq^{24n^2+72n+53}}{(1+q^{16n+22})^2}
        \nonumber\\&\quad+\frac{\left[X\left(q^{12};q^{16}\right)+X\left(-q^{22};q^{16}\right)\right]
            \times\left[1-X\left(q^{12};q^{16}\right)-X\left(-q^{22};q^{16}\right)\right]}{2}\nonumber\\&\quad-\sum_{n=-\infty}^\infty\left(    \frac{q^{16n+12}}{2(1-q^{16n+12})^2}-
        \frac{q^{16n+22}}{2(1+q^{16n+22})^2}\right)-\sum_{n=1}^\infty\frac{q^{16n}}{(1-q^{16n})^2}\nonumber\\&\quad-
        \frac{q^3[-q^{2};q^{16}]_\infty(q^{16};q^{16})^2_\infty}{[-q^{6},q^8;q^{16}]_\infty}\times\left[X\left(q^4;q^{16}\right)+X\left(-q^{22};q^{16}\right)\right].\nonumber
    \end{align}
    This together with \eqref{j1} implies \eqref{41}.

    Proceeding with the same steps as
     in the foregoing proof, we can
     get \eqref{42}--\eqref{41010}.
 \end{proof}

 We are now in a position to prove Theorem \ref{thrank}.
 \begin{proof}[Proof of Theorem \ref{thrank}]
    Lemma 2.1 of \cite{mao-2} gives that, for $1\leq b\leq k-1$,
    \begin{align}
        &\sum_{n=0}^{\infty}\left(NT(b,k,n)-NT(k-b,k,n)\right)q^n
        \nonumber\\&=\frac{k}{J_1}\sum_{\mbox{\tiny$\begin{array}{c} n=-\infty\\ n\neq0\end{array}$}}^\infty\frac{(-1)^{n}q^{\frac{n(3 n+1)}{2}+(b-1)n}(1-q^n) }{(1-q^{kn})^2}\nonumber\\&\quad-
        \frac{k-b}{J_1}\sum_{\mbox{\tiny$\begin{array}{c} n=-\infty\\ n\neq0\end{array}$}}^\infty\frac{(-1)^{n}q^{\frac{n(3 n+1)}{2}+(b-1)n}(1-q^n)}{1-q^{kn}}\label{gen7121}
        .
    \end{align}
    Setting $(b,k)=(2,8)$ in \eqref{gen7121}, one obtain
    \begin{align}
        &\sum_{n=0}^{\infty}\left(NT(2,8,n)
        -NT(6,8,n)\right)q^n
        \nonumber\\[6pt]&=
        \frac{8}{J_1}\sum_{\mbox{\tiny$
                \begin{array}{c} n=-\infty\\ n\neq0\end{array}$}}^\infty\frac{(-1)^{n}q^{\frac{n(3n+1)}{2}+n}(1-q^n) }{(1-q^{8n})^2}-
        \frac{6}{J_1}\sum_{\mbox{\tiny$
                \begin{array}{c} n=-\infty\\ n\neq0\end{array}$}}^\infty\frac{(-1)^{n}q^{\frac{n( 3n+1)}{2}+n}(1-q^n)}{1-q^{8n}}
        \nonumber\\[6pt]&=
        \frac{1}{J_1}\sum_{\mbox{\tiny$
                \begin{array}{c} n=-\infty\\ n\neq0\end{array}$}}^\infty\frac{(-1)^{n}q^{\frac{n(3n+1)}{2}+n}(1-q^n)\left\{8-6(1-q^{8n})\right\} }{(1-q^{8n})^2}.\nonumber
    \end{align}
    Invoking  $$\frac{4 }{(1-q^{8n})^2}=\frac{2-q^{4n} }{(1-q^{4n})^2}+\frac{2+q^{4n} }{(1+q^{4n})^2}$$
    and simplifying, we find that
    \begin{align}
        &\sum_{n=0}^{\infty}\left(NT(2,8,n)
        -NT(6,8,n)\right)q^n
        \nonumber\\[6pt]&=
        \frac{1}{2J_1}\sum_{\mbox{\tiny$
                \begin{array}{c} n=-\infty\\ n\neq0\end{array}$}}^\infty\frac{(-1)^{n}q^{\frac{n(3n+1)}{2}}\left\{2q^n-2q^{2n}-7q^{5n}+7q^{6n}+3q^{9n}-3q^{10n}\right\} }{(1-q^{4n})^2}\nonumber\\&\quad+\frac{1}{2J_1}\sum_{\mbox{\tiny$
                \begin{array}{c} n=-\infty\\ n\neq0\end{array}$}}^\infty\frac{(-1)^{n}q^{\frac{n(3n+1)}{2}}\left\{2q^n-2q^{2n}-5q^{5n}+5q^{6n}-3q^{9n}+3q^{10n}\right\} }{(1+q^{4n})^2}\nonumber
        \nonumber\\[6pt]&=
        \frac{1}{2J_1}\sum_{\mbox{\tiny$
                \begin{array}{c} n=-\infty\\ n\neq0\end{array}$}}^\infty\frac{(-1)^{n}q^{\frac{n(3n+1)}{2}}\left\{9q^n-9q^{2n}+3q^{9n}-3q^{10n}\right\} }{(1-q^{4n})^2}\nonumber\\&\quad+\frac{1}{2J_1}\sum_{
            n=-\infty}^\infty\frac{(-1)^{n}q^{\frac{n(3n+1)}{2}}\left\{7q^n-7q^{2n}-3q^{9n}+3q^{10n}\right\} }{(1+q^{4n})^2},\label{11201930}
    \end{align}
    where the second equality follows from
    \begin{align}
        \sum_{\mbox{\tiny$
                \begin{array}{c} n=-\infty\\ n\neq0\end{array}$}}^\infty\frac{(-1)^{n}q^{\frac{n(3n+1)}{2}+an} }{(1\pm q^{4n})^2}=\sum_{\mbox{\tiny$
                \begin{array}{c} n=-\infty\\ n\neq0\end{array}$}}^\infty\frac{(-1)^{n}q^{\frac{n(3n+1)}{2}+(7-a)n} }{(1\pm q^{4n})^2}.\nonumber
    \end{align}

    Substituting \eqref{41}--\eqref{41010} into \eqref{11201930} and simplifying, we arrive at
    \begin{align}
        &\sum_{n=0}^{\infty}\left(NT(2,8,n)
        -NT(6,8,n)\right)q^n
        \nonumber\\[6pt]&=
        -\frac{9}{2J_{16}}  \sum_{\mbox{\tiny$\begin{array}{c} n=-\infty\end{array}$}}^\infty\frac{(-1)^nq^{24n^2+72n+53}}{(1+q^{16n+22})^2}
        -
        \frac{9}{2J_{16}}   \sum_{\mbox{\tiny$\begin{array}{c} n=-\infty\end{array}$}}^\infty\frac{(-1)^nq^{24n^2+88n+75}}{(1+q^{16n+22})^2}
        \nonumber\\&-
        \frac{3}{2J_{16}}   \sum_{\mbox{\tiny$\begin{array}{c} n=-\infty\end{array}$}}^\infty\frac{(-1)^nq^{24n^2+104n+97}}{(1+q^{16n+22})^2}
        -
        \frac{3}{2J_{16}}   \sum_{\mbox{\tiny$\begin{array}{c} n=-\infty\end{array}$}}^\infty\frac{(-1)^nq^{24n^2+56n+31}}{(1+q^{16n+22})^2}
        \nonumber\\&
        -\frac{7}{2J_{16}}  \sum_{\mbox{\tiny$\begin{array}{c} n=-\infty\end{array}$}}^\infty\frac{(-1)^nq^{24n^2+72n+53}}{(1-q^{16n+22})^2}
        -
        \frac{7}{2J_{16}}   \sum_{\mbox{\tiny$\begin{array}{c} n=-\infty\end{array}$}}^\infty\frac{(-1)^nq^{24n^2+40n+15}}{(1-q^{16n+10})^2}
        \nonumber   \\& +
        \frac{3}{2J_{16}}   \sum_{\mbox{\tiny$\begin{array}{c} n=-\infty\end{array}$}}^\infty\frac{(-1)^nq^{24n^2+8n-15}}{(1-q^{16n-10})^2}
        -
        \frac{3}{2J_{16}}   \sum_{\mbox{\tiny$\begin{array}{c} n=-\infty\end{array}$}}^\infty\frac{(-1)^nq^{24n^2+24n-13}}{(1-q^{16n-6})^2}+L(q),\label{plq}
        \intertext{where}
        L(q):&=\Bigg\{  2X\left(-q^{12};q^{16}\right)\times\left(\frac{[-q^{6};q^{16}]_\infty }{[-1;q^{16}]_\infty}+\frac{q^{3}[-q^{2};q^{16}]_\infty }{[-q^{8};q^{16}]_\infty}\right)  \nonumber   \\&\qquad+
        2X\left(q^{22};q^{16}\right)\times\left(\frac{[-q^{6};q^{16}]_\infty }{[-1;q^{16}]_\infty}-\frac{q^{3}[-q^{2};q^{16}]_\infty }{[-q^{8};q^{16}]_\infty}\right)   \nonumber   \\&\qquad-
        \left(\frac{2[-q^{6};q^{16}]_\infty }{[-1;q^{16}]_\infty}-\frac{q^{3}[-q^{2};q^{16}]_\infty }{[-q^{8};q^{16}]_\infty}\right)
        \Bigg\}\times\frac{2[q^{4};q^{16}]_\infty J_{16}^2}{[q^{6},-q^{4};q^{16}]_\infty J_1}.\nonumber
    \end{align}
    Note that none of the $q$-expansion of the series on the right side of \eqref{plq} (except $L(q)$) contains terms of the form $q^{2n}$.
    We only need to study the $2$-dissection for $L(q)$.
    Invoking \eqref{x4}, \eqref{x22}, \eqref{zp1} and collecting terms with even exponents, we prove \eqref{thn8}.
 \end{proof}

\section{Proof of Theorem \ref{thmain}}
We rewrite \eqref{v-1} and \eqref{v-2} as follows:
\begin{align}\label{v-11}
&   \sum_{n\geq 0} (M_{\omega}(1,4,4n) -M_{\omega}
    (3,4,4n))q^n
    =f_1(q)+f_2(q),
    \intertext{with}
    &   f_1(q):=\frac{1}{4J_1}A_0(q)B_0(q)
    +\frac{q}{4J_1}
    A_2(q)B_2(q),\nonumber\\[6pt]
&   f_2(q):=-\frac{1}{4J_1}A_0(q)B_0(q)\varphi(q)^2
    -\frac{q}{J_1}\bigg(
    \frac{1}{4}A_2(q)B_2(q)\varphi(q)^2
    +A_2(q)B_1(q)\psi(q)^2 \nonumber\\[6pt]
    &\qquad\qquad
    -(A_0(q)B_2(q)
    +A_2(q)B_0(q))\psi(q^2)^2 \bigg)
    -\frac{q^2}{J_1}A_0(q)
    B_3(q)\psi(q)^2\nonumber
\intertext{and}
\label{v-22}
    &\sum_{n\geq 0} (M_{\omega}(1,4,4n+2)-M_{\omega}
    (3,4,4n+2))q^n
    =f_3(q)+f_4(q)
    \intertext{with}
&   f_3(q)  =- \frac{1}{4J_1}(A_0(q)B_2(q)
    +A_2(q)B_0(q)),\nonumber\\
&       f_4(q)  =\frac{1}{4J_1}(A_0(q)B_2(q)
    +A_2(q)B_0(q))\varphi(q)^2
    +\frac{1}{J_1}A_0(q)(
    B_1(q)\psi(q)^2 -B_0(q)\psi(q^2)^2)
    \nonumber\\[6pt]
    &\qquad\quad -\frac{q}{J_1}A_2(q)B_2(q)\psi(q^2)^2
    +\frac{q^2}{J_1}A_2(q)B_3(q)\psi(q)^2.\nonumber
\end{align}
Applying \eqref{thn8} \eqref{v-11} and \eqref{v-22}, we find that Theorem \ref{thmain} is implied by
\begin{align}
    R_1(q)+R_2(q)=f_1(q^2)+f_2(q^2)-q(f_3(q^2)+f_4(q^2)).\nonumber
    \end{align}
Thus, it suffices to show that
\begin{align}
    R_1(q)=f_1(q^2)-qf_3(q^2),\label{rf1}\\
R_2(q)=f_2(q^2)-qf_4(q^2).\label{rf2}
\end{align}
Multiplying by $\frac
{J_{3,16}^3J_{4,16}J_{5,16}^{3}J_{8,16}^{13/2}}{q^6J_{16}^{29/2}}$ on both sides of
\eqref{rf1} and simplifying, we find that it is equivalent to
\begin{align}
&   \frac
    {J_{6,16}^{2}J_{8,16}^{8}}{4q^6J_{1,16}^2J_{3,16}^2J_{4,16}^2J_{5,16}^2J_{7,16}^2}
-   \frac
    {J_{2,16}^{2}J_{8,16}^{6}}{2q^4J_{1,16}^4J_{7,16}^4}
\nonumber\\&    -   \frac
    {J_{3,64}^3J_{5,64}^3J_{8,64}^{12}J_{11,64}^3J_{12,64}J_{13,64}^3
        J_{19,64}^3J_{20,64}^2J_{21,64}^3J_{24,64}^{12}J_{27,64}^3J_{29,64}^3}{4q^6J_{10,64}J_{16,64}J_{22,64}J_{64}^{48}}
\nonumber   \\& -   \frac
    {J_{3,64}^3J_{4,64}J_{5,64}^3J_{8,64}^{12}J_{10,64}J_{11,64}^3J_{12,64}J_{13,64}^3
        J_{19,64}^3J_{21,64}^3J_{22,64}J_{24,64}^{12}J_{27,64}^3J_{28,64}J_{29,64}^3}
    {4q^4J_{2,64}J_{14,64}J_{16,64}J_{18,64}J_{30,64}J_{64}^{48}}
    \nonumber   \\&
    -   \frac
    {J_{3,64}^3J_{5,64}^3J_{8,64}^{12}J_{11,64}^3J_{12,64}^2J_{13,64}^3
        J_{19,64}^3J_{20,64}J_{21,64}^3J_{24,64}^{12}J_{27,64}^3J_{29,64}^3}
    {4q^5J_{6,64}J_{16,64}J_{26,64}J_{64}^{48}}
    \nonumber   \\& -   \frac
    {J_{3,64}^3J_{4,64}J_{5,64}^3J_{6,64}J_{8,64}^{12}J_{11,64}^3J_{13,64}^3
        J_{19,64}^3J_{20,64}J_{21,64}^3J_{24,64}^{12}J_{26,64}J_{27,64}^3J_{28,64}J_{29,64}^3}
    {4q^5J_{2,64}J_{14,64}J_{16,64}J_{18,64}J_{30,64}J_{64}^{48}}=0\label{rf11}.
\end{align}
Using \cite[Theorem 3]{sinai}, we verify that
each term on the left side of \eqref{rf11} is a modular function on
$\Gamma_{1}(64)$. Then we can prove \eqref{rf11} with
the MAPLE package {\it thetaids} \cite{maple}. For the Maple commands,
 see https://github.com/dongpanghu/Code2/blob/main/code.md.
 This proves \eqref{rf1}.
With a completely similar argument, one can obtain \eqref{rf2} and the detailed proof is omitted.
Then the proof of Theorem \ref{thmain} is completed.

 \vspace{0.5cm}

\noindent{\bf Acknowledgments.}
 This work was partially
supported by National Natural Science Foundation of
  China (12071331,
 11971341 and  11971203) and
     the Natural Science Foundation of
   Jiangsu Province of China (BK20221383).

\end{document}